\documentclass{amsart}
\usepackage{amsmath, amsfonts, amssymb,amsthm}
\usepackage{amstext}
\usepackage{mathrsfs}
\usepackage{comment}

\theoremstyle{definition}
\newtheorem{definition}{Definition}[section]

\newtheorem{theorem}[definition]{Theorem}
\newtheorem{lemma}[definition]{Lemma}
\newtheorem{proposition}[definition]{Proposition}
\newtheorem{corollary}[definition]{Corollary}

\newtheorem*{tha}{Theorem A}
\newtheorem*{thb}{Theorem B}
\newtheorem*{thc}{Theorem C}
\newtheorem*{thd}{Theorem D}
\newtheorem*{theo}{Theorem E}

\newtheorem*{thm}{Main Theorem}
\newtheorem*{cro}{Corollary}

\newcommand{\sq}[1]{\ifx#1([\else\ifx#1)]%
  \else\message{invalid use of "sq"}\fi\fi}

\newcommand{\PP}{\mathbb{P}}

\newcommand{\RR}{\mathbb{R}}

\newcommand{\CC}{\mathbb{C}}
\newcommand{\QQ}{\mathbb{Q}}

\DeclareMathOperator{\Supp}{Supp}

\DeclareMathOperator{\Nev}{Nev}

\DeclareMathOperator{\Nevbir}{Nev_{\operatorname{bir}}}

\DeclareMathSymbol{\idot}{\mathbin}{operators}{`\.}
\allowdisplaybreaks
\newcommand{\SO}{\mathcal{O}}

\newcommand{\SI}{\mathcal{I}}

\newcommand{\Proj}{\mathrm{Proj}}
\newcommand{\supp}{\mathrm{Supp}}

\begin{document}
\title[Generalized Subspace Theorem For Closed Subschemes in Subgeneral Position]{A Generalized Subspace Theorem For Closed Subschemes in Subgeneral Position}
\author{Yan He}
\address{Department of Mathematics\newline
\indent University of Houston\newline
\indent Houston,  TX 77204, U.S.A.} 
\email{yanhe@math.uh.edu}
\author{Min Ru}
\address{Department of Mathematics\newline
\indent University of Houston\newline
\indent Houston,  TX 77204, U.S.A.} 
\email{minru@math.uh.edu}

\begin{abstract}
In this paper, we extend the  recent theorem of G. Heier and A. Levin \cite{HL17} on 
the  generalization of Schmidt's subspace theorem and Cartan's Second Main Theorem in Nevanlinna theory to 
closed subschemes located  in $l$-subgeneral position, using the generic linear combination technique due to Quang
 (see  \cite{Quang19}).
\end{abstract}
 \thanks{2010\ {\it Mathematics Subject Classification.}
11J87, 11J97, 11J13,  14C20,    32H30.}  
\keywords{Subspace theorem, Second Main Theorem, subgeneral position}
\thanks{The second named author is supported in part by Simon Foundation grant award \#531604.}

\baselineskip=16truept \maketitle \pagestyle{myheadings}
\markboth{}{A generalized subspace theorem  for closed subschemes in subgeneral position}
\section{Introduction}
In  recent years, there have been many efforts and developments in extending the 
Schmidt's subspace theorem, as well as  Cartan's Second Main Theorem in Nevanlinna theory.
The break through was made 
 by Evertse-Ferretti \cite{ef_festschrift} by proving the following Theorem A (see also the result of Corvaja and Zannier \cite{cz_ajm}).
 We use the standard notations in Nevanlinna theory and Diophantine approximation
(see, for example, \cite{vojta_cm}, \cite{Vojta_LNM}, \cite{book:minru}, \cite{R4} and \cite{ru}).  For the closed subschemes, we use the notations in  \cite{HL17} and \cite{rw}.
 
\begin{tha}[\cite{ef_festschrift}]\label{thm:EF}
Let $X$ be a projective variety of dimension $n$ defined over a number field $k$. Let $S$ be a finite set of places of $k$.
For each $v\in S$, let $D_{0, v}, \dots, D_{n, v}$ be  effective Cartier divisors on $X$, defined over $k$, and in general position. Suppose that there exists an ample Cartier
divisor $A$ on $X$ and positive integers $d_{j, v}$ 
such that $D_{j, v}$ is linearly equivalent to $d_{j, v}A$  (which we denote by $D_{j, v}\sim d_{j,v} A$) for  $j=0,\dots, n$ and all $v\in S$. 
Then, for every $\epsilon>0$, there exists a proper Zariski-closed subset $Z\subset X$ such that for all points $x\in X(k)\setminus Z$,
$$\sum_{v\in S}\sum_{j=0}^{n}\frac{1}{d_{j, v}}\lambda_{D_{j, v}, v}(x)
  \leq (n+1+\epsilon)h_A(x).$$ 
\end{tha}
Here, $\lambda_{D_{j, v},v}$ is a local Weil function associated to the divisor $D_{j, v}$ and place $v$ in $S$. 

The counterpart result in Nevanlinna theory was obtained by the  second named author \cite{ru_annals} in 2009. 
\begin{thb}[\cite{ru_annals}]
Let $X$ be a complex projective variety and 
$D_1, \dots, D_q$ be  effective Cartier divisors on $X$, located in general position.
Suppose that there exists an ample Cartier divisor $A$ on $X$
and positive integers $d_j$ such that    $D_j\sim d_j A$  for
 $j=1, \dots, q$.
Let $f: \mathbb C \rightarrow X$ be a
 holomorphic map with Zariski dense image. 
 Then, for every $\epsilon > 0$, 
$$ \sum_{j=1}^q {\frac{1}{d_j}} m_f(r, D_j) \le_{exc} (\dim X+1+\epsilon)T_{f,A}(r)$$
where $\leq_{exc}$ means the inequality holds for all $r\in \RR^{\geq 0}$ outside a set of finite Lebesgue measure.
\end{thb}
In order to formulate and uniformize the results for a general divisor $D$ on $X$,  the second named author 
introduced (see \cite{R4} and \cite{ru}) the notion of \emph{Nevanlinna constant} $\Nev(D)$. 
  Later, the Nevanlinna constant $\Nev(D)$ was further developed by Ru-Vojta \cite{rv} to the
\emph{birational Nevanlinna constant} $\Nevbir(L, D)$, where $L$ is a line sheaf (i.e. an invertible sheaf) over $X$.
With the notation  $\Nevbir(L, D)$, the following result  was obtained  by Ru-Vojta in \cite{rv}.

\begin{thc}[Arithmetic Part, \cite{rv}] \label{b_thmc}
Let $k$ be a number field, and $S$ be a finite set of places of $k$.
Let $X$ be a projective variety and
 let $D$ be an effective Cartier
divisor on $X$, both defined over $k$.
Let $L$ be a line sheaf on $X$ with $\dim H^0(X, L^N)\ge 1$
for some $N>0$. 
Then, for every $\epsilon>0$,
there is a proper Zariski-closed subset $Z$ of $X$ such that the inequality
$$
m_S(x,D)  \le \left(\Nevbir(L, D)+\epsilon\right) h_{L}(x)
$$
holds for all $x\in X(k)$ outside of $Z$.  Here 
$m_S(x,D):=\sum_{v\in S}\lambda_{D,v}(x).$\end{thc}

The  corresponding result in  Nevanlinna theory is also obtained in \cite{rv}. Here and in the following context,  we only state the arithmetic results.

In general  $\Nevbir(L, D)$ is hard to compute.  However, in the case when $D=D_1+\cdots + D_q$ where  $D_1, \dots, D_q$ are  effective Cartier divisors intersecting properly
on $X$, Ru-Vojta \cite{rv} computed the  $\Nevbir(L, D)$ in terms of the  following so-called $\beta$-constant.

\begin{definition}\label{def_aut_lambda}
Let $L$ be a big line sheaf and let $D$ be a nonzero effective
Cartier divisor on a complete variety $X$.  We define
$$
  \beta( L, D)
    = \liminf_{N\to\infty}
      \frac {\sum_{m\ge 1}h^0(L^N(-mD))} {Nh^0(L^N)}.
$$
\end{definition}

\begin{thd}[Arithmetic Part, \cite{rv}] \label{Ga}
Let $X$ be a projective variety, and $D_1, \dots, D_q$ be effective Cartier divisors, both defined over a number field $k$. Assume that 
$D_1, \dots, D_q$ intersect properly on $X$. Let $S\subset M_k$ be a finite set of places on $k$.   
 Let $L$ be a big line sheaf on $X$.
Then, for every $\epsilon>0$, there is a proper Zariski-closed subset $Z$ of $X$ such that the inequality 
\begin{equation} \label{Ga_ineq}
\sum_{j=1}^q \beta(L,D_j) m_S(x,D_j)
    \leq (1+\epsilon) h_{L}(x)
\end{equation}
holds for all $x\in X(k)$ outside of $Z$.
\end{thd}

On the other hand, in \cite{HL17}, Heier-Levin   (see also McKinnon-Roth \cite{McRo}) obtained a similar statement where $\beta(L,D_j)$ is replaced by the Seshadri constant $\epsilon_{D_j}(L)$ in the case when $L$ is ample (note: Although \cite{HL17} was listed in 2017, the statement they obtained, as noted in \cite{HL17}, indeed goes back to an earlier unpublished manuscript which they presented at 2014 meeting in Banff on Vojta's Conjectures). In addition,  they obtained the results for closed subschemes as well.

\begin{definition}
Let $X$ be a projective variety,
$Y$ be a closed subscheme of $X$. Let $A$ be a nef Cartier divisor on $X$. The {\it Seshadri constant $\epsilon_{Y}(A)$ of $Y$ with respect to $A$ }is a real number given by 
$$\epsilon_{Y}(A) = \sup\{\gamma\in \QQ: \pi^*A-\gamma E \text{ is nef}\}$$
where $\pi: \tilde X\rightarrow X$ is the blowing-up of $X$ along $Y$ with the exceptional divisor $E$. Note that when $Y$ is a Cartier divisor on $X$, then
$\epsilon_{Y}(A)$ can be defined similarly without doing the blowing-up.
\end{definition}
\begin{theo}  [Arithmetic Part, \cite{HL17}]
Let $X$ be a projective variety of dimension $n$ defined over a number field $k$. Let $S$ be a finite set of places of $k$. For each $v\in S$, let $Y_{0,v},\dots,Y_{n,v}$ be closed subschemes of $X$, defined over $k$, and in general position. Let $A$ be an ample Cartier divisor on $X$. Then, for  $\epsilon > 0$, 
 there exists a proper Zariski-closed subset $Z\subset X$ such that for all points $x\in X(k)\setminus Z$,
$$\sum_{v\in S}\sum_{j=0}^n \epsilon_{Y_{j,v}}(A) \lambda _{Y_{j,v},v}(x) < (n+1+\epsilon)h_A(x).$$
\end{theo}

Note that, as being indicated in Ru-Vojta \cite{rv},  when $D$ is an effective Cartier divisor, 
$\beta(A,D) \ge \frac{\epsilon_D(A)}{n+1}$, so the above result in the divisors case is 
 indeed a consequence of the result of Ru-Vojta (Note that Theorem D can also be stated in the same form of above).   However, the proof 
of their result is much simpler than Ru-Vojta, directly deriving from Theorem A.

This paper extends the above mentioned result of  Heier-Levin, i.e. Theorem E,  to the case where the subschemes are in $l$-subgeneral position. The analytic case is also obtained.
When $l=n$, it is just the result of Heier-Levin.  
We basically follow the argument of  Heier-Levin \cite{HL17}, together with 
technique due to Quang (\cite{Quang19}) which allow us to deal with the $l$-subgeneral position case.
Note that the problem of extending 
  Theorem D  to $l$-subgeneral position still remains open.  
 
 The following is our main result. 
 \begin{thm} [Arithmetic Part] 
 Let $X$ be a projective variety of dimension $n$ defined over a number field $k$. Let $S$ be a finite set of places of $k$.
   For each $v\in S$, let $Y_{0,v},\dots,Y_{l,v}$ be closed subschemes of $X$, defined over $k$, and in $l$-subgeneral position with $l\ge n$. 
  Let $A$ be an ample Cartier divisor on $X$. Then, for  $\epsilon > 0$, 
 there exists a proper Zariski-closed subset $Z\subset X$ such that for all points $x\in X(k)\setminus Z$,
$$\sum_{v\in S}\sum_{j=0}^l \epsilon_{Y_{j,v}}(A) \lambda _{Y_{j,v},v}(x) < [(l-n+1)(n+1)+\epsilon]h_A(x).$$
 \end{thm}
 Combining this with Lemma 3.4 in \cite{rw}, it gives the the following Corollary. 
 
 \begin{cro}
 Let
 $X$ be a projective variety of dimension $n$ defined over a number field $k$. Let $S$ be a finite set of places of $k$. Let $Y_1,\dots,Y_q$ be closed subschemes of $X$ defined over $k$, in $l$-subgeneral position with $l\ge n$. Let $A$ be an ample Cartier divisor on $X$. Then, for  $\epsilon > 0$, there exists a Zariski-closed set $Z$ of $X$ such that for all $x\in X(k)\setminus Z$,
$$\sum_{j=1}^q  \epsilon_{Y_j}(A)m_S(x,Y_j) \leq [(l-n+1)(n+1)+\epsilon]h_A(x).$$
\end{cro}

 \begin{thm} [Analytic  Part]  Let $X$ be a complex  projective variety of dimension $n$. Let $Y_1,\dots,Y_q$ be closed subschemes of $X$ in $l$-subgeneral position with $l\ge n$.
 Let $A$ be an ample Cartier divisor on $X$.
 Let $f:\CC\to X$ be a holomorphic curve with Zariski-dense image. Then, for every $\epsilon>0$,
$$\sum_{j=1}^q \epsilon_{Y_j}(A) m_f(r,Y_j) \leq_{exc} [(l-n+1)(n+1) + \epsilon]T_{f,A}(r).$$
\end{thm}
The basic set ups of arithmetic part for closed subschemes can be found in \cite{Sil}, the analytic analogue of which can be found in \cite{Yam04}.
In this paper, to make the results clear, we deal with the divisor case
(i.e.  codim $Y=1$) and the
 proper closed subscheme case (codim $Y>1$)  separately. 
The divisor case is a special case of the main theorem without blowing-ups.

\section{The divisor case}
For a number field $k$, recall that $M_k$ denotes the set of places of $k$, and
that $k_\upsilon$ denotes the completion of $k$ at a place $\upsilon\in M_k$.
Norms $\|\cdot\|_\upsilon$ on $k$ are normalized so that
$$\|x\|_\upsilon = |\sigma(x)|^{[k_\upsilon:\mathbb R]} \qquad\text{or}\qquad
  \|p\|_\upsilon = p^{-[k_\upsilon:\mathbb Q_p]}$$
if $\upsilon\in M_k$ is an Archimedean place corresponding to an embedding
$\sigma\colon k\hookrightarrow\mathbb C$ or a non-Archimedean place lying
over a rational prime $p$, respectively.

An $M_k$-constant is a collection $(c_v)_{v\in M_k}$ of real constants such
that $c_v=0$ for all but finitely many $v$.
Heights are logarithmic and relative to the number field used as a base field which is always denoted by $k$.  For example, if $P$ is a point
on $\mathbb P^n_k$ with homogeneous coordinate $[x_0:\dots:x_n]$ in $k$, then
$$h(P) = h_{\SO_{\PP^n_k}(1)}(P)
  = \sum_{\upsilon\in M_k} \log\max\{\|x_0\|_\upsilon,\dots,\|x_n\|_\upsilon\}
  \;.$$

Let $X$ be a non-singular projective variety over $\mathbb C$, let $D$ be an effective Cartier divisor
on $X$, and let $s$ be a canonical section of $\SO(D)$ (i.e.,
a global section $s$ such that $(s=0)=D$).  Choose a Hermitian metric $\|\cdot\|$
on $\SO(D)$.  In Nevanlinna theory, one often encounters the
function
\begin{equation}\label{eq_weil_intro}
  \lambda_D(x) := -\log\|s(x)\|\;;
\end{equation}
this is a real-valued function on $X(\mathbb C)\setminus\Supp D$.
It is linear in $D$ (over a suitable domain), so by linearity and continuity
it can be extended to a definition of $\lambda_D$ for a general Cartier divisor
$D$ on $X$.

Weil functions are counterparts to such functions in number theory.
For its definition and detailed properties,  see \cite[Ch.~10]{lang} or \cite[Sect.~8]{vojta_cm}.  For reader's convenience we list some properties of Weil functions for varieties and divisors defined over $k$ (the complex case is similar).
\begin{proposition}(see \cite{rv})\label{weil_eff} Let $k$ be a number field. 
Let $X$ be a normal complete variety, let $D$ be a Cartier divisor on $X$, both defined over $k$. 
Let $\lambda_D$ be a Weil function for $D$.  Then the following
conditions are equivalent.
\begin{enumerate}
\item $D$ is effective.
\item $\lambda_D$ is bounded from below by an $M_k$-constant.
\item for all $v\in M_k$, $\lambda_{D,v}$ is bounded from below.
\item there exists $v\in M_k$ such that $\lambda_{D,v}$ is bounded from below.
\end{enumerate}
\end{proposition}

\begin{corollary}
If $D_1-D_2$ is effective. Then $\lambda_{D_1}\geq \lambda_{D_2}$ up to an $M_k$-constant.
\end{corollary}

\begin{proposition}(see \cite[Theorem 8.8]{vojta_cm})
Let $X$ be a complete variety over a number field $k$.

(a) If $D_1,D_2$ are Cartier divisors on $X$, let $\lambda_{D_1},\lambda_{D_2}$ be Weil functions for $D_1,D_2$ respectively, then $\lambda_{D_1} + \lambda_{D_2}$ extends uniquely to a Weil function for $D_1+D_2$.

(b) If $\lambda$ is a Weil function for a Cartier divisor $D$ on $X$, and if $f:X'\to X$ is a morphism of $k$-varieties such that $f(X')\not\subset \Supp D$, then $x\to \lambda(f(x))$ (defined on $X'\setminus \Supp f^*D$) is a Weil function for the Cartier divisor $f^*D$ on $X'$, which we denote by $\lambda_{f^*D}$.
\end{proposition}

Let $X$ be a  complex  projective variety and $L\rightarrow X$ be a positive line bundle. 
Denote by $\|\cdot\|$ a Hermitian metric in $L$ and by $\omega$ its Chern form. 
Let $f: {\mathbb C}\rightarrow X$ be a holomorphic map. We define
$$T_{f, L}(r)=\int_1^r {dt\over t} \int_{|z|<t} f^*\omega,$$
and call it the {\it characteristic} (or {\it height}) function of $f$ with respect to $L$. It is independent of, up to bounded term,
the choices of the metric on $L$.  The definition can be extended to arbitrary line bundle, since any line bundle 
$L$ can be written as $L=L_1\otimes L^{-1}_2$ with $L_1, L_2$ are both positive.
For an effective divisor on $X$, we define a function of $r$,  for any holomorphic map $f: {\mathbb C}\rightarrow X$ with $f(\CC)\not\subset \Supp D$,
$$m_f(r, D):=\int_0^{2\pi}\lambda_D (f(re^{i\theta})){d\theta\over 2\pi}.$$

Let $X$ be a complete variety over a number field $k$, and $D$ be a Cartier divisor on $X$.
The height $h_D(x)$ for points $x\in X(\bar{k})$ is defined, up to $O(1)$, as 
$$h_{D}(x)={1\over [F:k]}\sum_{w\in M_F} \lambda_{D,w}(x)$$
for any number field $F\subset k(x)$. It is independent, up to $O(1)$, of the choice of $F$, as well as the choice of 
Weil functions (see    \cite{vojta_cm}, Page 143-144). 
 If $L$ is a line sheaf on $X$, then we define $h_{L}(x)=h_{D}(x)$ for points $x\in X(\bar{k})$, where $D$ is any Cartier divisor
for which ${\mathcal O}(D)\cong L$. Again, it is only defined up to $O(1)$.  This definition agrees with the definition given earlier in the 
case when $X$ is projective. Let $S\subset M_k$ be a finite set of places on $X$. We have the function $m_S(x,D)$ defined  in the previous section as the analogy of $m_f(r,D)$ in Diophantine geometry.

\begin{definition} Let $V$ be a  normal projective variety 
and $X\subset V$ be an irreducible normal subvariety of dimension $n$. The Cartier divisors $D_1,\dots,D_q$ on $V$ are called {\it in $l$-subgeneral position on $X$} if for every choice $J\subset \{1,\dots,q\}$ with $\# J\leq l+1$, 
$$\dim \left((\bigcap _{j\in J} \Supp D_j)\cap X\right) \leq l-\#J.$$
When $V=X$, then we say that the divisors $D_1,\dots,D_q$ on $X$ are in $l$-subgeneral position.
Similar definition applies for subschemes $Y_1, \dots, Y_q$. 
\end{definition}
The notion ``in $n$-subgeneral position" means the same as "in general position" on a variety of dimension $n$.  In the rest of the paper, by convenience, we just write $D_1\cap D_2$ to denote $\supp D_1\cap \supp D_2$. 

The following is a reformulation of Theorem 1.22 in \cite{Sha94} which plays important role in the arguments regarding dimensions.

\begin{proposition}\label{S}
Let $X$ be a projective variety, and $A$ be an ample Cartier divisor on $X$. Let $F\subset X$ be a proper irreducible subvariety. Then either $F\subset A$ or $\dim (F\cap A) \leq \dim F - 1$.
\end{proposition}

In this section, we prove the following theorem.
\begin{theorem}[Arithmetic Result in the Divisors case]\label{thm:divisorcase}
Let $X$ be a projective variety of dimension $n$ defined over a number field $k$.  Let $S$ be a finite set of places of $k$.
  For each $v\in S$, let $D_{0,v},\dots,D_{l,v}$ be Cartier divisors on  $X$, defined over $k$, and in $l$-subgeneral position with $l\ge n$.  
  Let $A$ be an ample Cartier divisor on $X$. Then, for  $\epsilon > 0$, 
 there exists a proper Zariski-closed subset $Z\subset X$ such that for all points $x\in X(k)\setminus Z$,
$$\sum_{v\in S}\sum_{j=0}^l \epsilon_{D_{j,v}}(A) \lambda _{D_{j,v},v}(x) < [(l-n+1)(n+1)+\epsilon]h_A(x),$$
where $\epsilon_{D_{j, v}}(A)$ are the Seshadri constants of $D_{j, v}$ with respect to $A$. 
\end{theorem}

We give a reformulation of a key lemma which is originally due to Quang (see Lemma 3.1 in \cite{Quang19}). We include a proof here for the reader's convenience.
\begin{lemma}[\cite{Quang19}, reformulated]\label{quang}
Let $k$ be a number field. Let $X\subset \PP^M_k$ be an irreducible projective variety of dimension $n$. Let $H_1,\dots, H_{l+1}$ be hyperplanes in $\PP^M_k$ which are in $l$-subgeneral position on $X$ with $l\ge n$. Let $L_1,\dots,L_{l+1}$ be the normalized linear forms defining $H_1,\dots,H_{l+1}$ respectively. Then there exist linear forms $L_1',\dots,L_{n+1}'$ of $M+1$ variables such that 

(a) $L_1' = L_1$.

(b) For every $t\in \{2,\dots,n+1\}$, $L'_t \in \mathrm{span}_k(L_2,\dots,L_{l-n+t})$  i.e., $L'_t$ is a $k$-linear combination of $L_2,\dots, L_{l-n+t}$.

(c) Let $H_j'$, $j = 1,\dots, n+1$, be the hyperplanes defined by $L'_j$, $j = 1,\dots,n+1$. Then they are in general position on $X$.
\end{lemma}

\begin{proof}
Let $H_1' = H_1.$ Let $W_{l-n+2}$ denote the $k$-vector space  spanned by $L_2,\dots,L_{l-n+2}$. For each irreducible component $\Gamma$ 
of $H_1 \cap X$ with $\dim \Gamma=n-1$,  let 
$V_{\Gamma}\subset W_{l-n+2}$ be the linear subspace of $W_{l-n+2}$ consisting of $k$-linear combinations of $L_2,
\dots,L_{l-n+2}$ that vanish entirely on $\Gamma$.  Since $H_1,\dots,H_{l+1}$ are in $l$-subgeneral position on $X$, 
$$\dim(H_1\cap H_2 \cap \dots \cap H_{l-n+2}\cap X)\leq l-(l-n+2)=n-2 <\dim \Gamma,$$
there must be some $H_j$ with $2\leq j\leq l-n+2$ such that $\Gamma \not\subset H_j$, so $L_j$ is not in $V_{\Gamma}$.
 Hence $V_{\Gamma}$ is a proper subspace of $W_{l-n+2}$. Observe that the choices of $\Gamma$ is finite, we thus have
 $$W_{l-n+2}\setminus \bigcup _{\Gamma}V_{\Gamma}\neq \emptyset,$$
where the union is taken over all irreducible component $\Gamma$ of $H_1\cap X$ with $\dim \Gamma = n-1$. 
Take an $L_2' \in W_{l-n+2}\setminus \bigcup _{\Gamma}V_{\Gamma}$. Then $L_2'$ does not vanish on any
 irreducible component $\Gamma$  of $H_1'\cap X$ with $\dim \Gamma=n-1$. Let $H_2'$ be the hyperplane defined by $L_2'$. Then, by Proposition \ref{S},  
 $\dim(H_1'\cap H_2'\cap X)= \dim(H_1\cap H_2'\cap X)\leq n-2$. Hence $H_1',H_2'$ are in general position on $X$.
 
Now consider irreducible components $ \Gamma' $  of $(H_1'\cap H_2'\cap X)$ with $\dim \Gamma'=n-2$.  
Let $W_{l-n+3}$ be the $k$-vector space spanned by $L_2,\dots,L_{l-n+3}$ and 
 $V_{\Gamma'}\subset W_{l-n+2}$ be the $k$-linear subspace of $W_{l-n+3}$ given by all $k$-linear combinations of $L_2,\dots, L_{l-n+3}$ that vanishes entirely on $\Gamma'$. 
Similar to the argument above, for every such $\Gamma'$, 
$V_{\Gamma'}$ is proper linear subspace of $W_{l-n+3}$ since $\dim(H_1\cap H_2\cap H_3,\dots,H_{l-n+3})\leq n-3 < \dim \Gamma' = n-2$. Thus
$$W_{l-n+3}\setminus \bigcup_{\Gamma'}V_{\Gamma'}\neq \emptyset.$$
Take an  $L_3'\in W_{l-n+3}\setminus \bigcup_{\Gamma'}V_{\Gamma'}$. By the construction $L_3'$ does not vanish on any irreducible component $\Gamma'$  of $(H_1'\cap H_2'\cap X)$ with $\dim \Gamma'=n-2$. Let $H_3'$ be the hyperplane defined by $L'_3$, then by Proposition \ref{S}, $\dim(H_1'\cap H_2'\cap H_3'\cap X)\leq n-3$, i.e. 
$H_1',H_2',H_3'$ are in general position on $X$.

Repeating the argument we obtain $H_1',\dots,H_{n+1}'$ which are in general position on $X$.  
\end{proof}

\noindent{\it Proof of Theorem \ref{thm:divisorcase}}. 
Fix  $\epsilon > 0$.
 Choose rational number $\delta > 0$ such that 
 $$\delta(l-n+1)+\delta(l-n+1)(n+1+\delta)<\epsilon.$$ 
Then, for small enough positive rational number $\delta'$ depending on $\delta$, 
 $\delta A - \delta'D_{i,v}$ is  $\QQ$-ample for all $i=0, \dots, l$ and $v\in S$. We fix such $\delta'$. 
By the definition of the Seshadri constant, there exists a rational number $\epsilon_{i, v}>0$ such that 
\begin{equation}\label{a1}
  \epsilon_{D_{i,v}}(A)-\delta'\leq\epsilon_{i,v}\leq\epsilon_{D_{i, v}}(A)
  \end{equation}
  and that  $A-\epsilon_{i,v}D_{i, v}$ is  $\QQ$-nef for all $i=0, \dots, l$ and $v\in S$.  
     With such choices of $\delta,\delta'$ and $\epsilon_{i,v}$, we have   $(1+\delta)A-(\epsilon_{i,v}+\delta')D_{i,v}$ is $\QQ$-ample. 
Take natural number $N$ large enough such that $N(1+\delta)A$ and $N[(1+\delta) A -(\epsilon_{i,v}+\delta')D_{i, v}]$ become very ample integral divisors for all
  $i=0, \dots, l$ and $v\in S$.  
  
  Fix $v\in S$, we claim that {\it there are  sections $$s_{i,v}\in H^0(X, N(1+\delta)A - N(\epsilon_{i, v}+\delta')D_{i, v}),  ~i=0, \dots, l,$$ 
   such that their divisors 
  $\mathrm{div}(s_{i,v}), i=0, \dots, l$, are 
  in      $l$-subgeneral position on $X$.}
  Here we regard 
  $H^0(X, N(1+\delta)A - N(\epsilon_{i, v}+\delta')D_{i, v})$  as a subpsace of $H^0(X, N(1+\delta)A)$.
   We prove the claim by induction. 
   Assume that, for some $j\in \{0, \dots, l\}$, 
sections $s_{0, v}, \dots, s_{j-1, v}$ with the desired property have been found and  
$\mathrm{div}(s_{0, v}), \dots,  \mathrm{div}(s_{j-1, v}), D_{j,v}, \dots, D_{l,v}$ are in $l$-subgeneral position (for $j=0$, this reduces to the hypothesis that $D_{0, v}, \dots, D_{l, v}$ are 
in $l$-subgeneral position). 
In particular, it gives
\begin{equation}\label{eqn:dim1}
\dim\left((\cap_{i\in I} \mathrm{div}(s_{i, v}))\cap (\cap_{j\in J}D_{j, v})\right)
\leq l-\#I-\#J,
\end{equation}
and
\begin{equation}\label{eqn:dim2}
\dim \left(D_{j,v}\cap (\cap_{i\in I}\mathrm{div}(s_{i,v}))\cap (\cap_{j\in J}D_{j,v})\right) \leq l-\#I-\#J-1. 
\end{equation}
for any subset $I\subset \{0, \dots, j-1\}$ and $J\subset \{j+1, \dots, l\}$. 
 We choose $s_{j,v}\in H^0(X, N(1+\delta)A - N(\epsilon_{j, v}+\delta')D_{j, v})$
such that $s_{j,v}$ does not vanish entirely on any irreducible components of $(\cap_{i\in I}  \mathrm{div}(s_{i, v}))\cap (\cap_{j\in J}D_{j, v})$ where not both $I$ and $J$ are empty. Then 
by 
Proposition \ref{S}, 
\begin{equation}\label{eqn:dim3} 
\dim \left( \mathrm{div}(s_{j,v})\cap (\cap_{i\in I} (\mathrm{div}(s_{i,v}))\cap (\cap_{j\in J}D_{j,v})\right) \leq l-\#I-\#J-1.
\end{equation}
  This means that $\mathrm{div}(s_{0, v}), \dots,  \mathrm{div}(s_{j-1, v}), \mathrm{div}(s_{j,v}), D_{j+1,v},\dots, D_{l,v}$ are in $l$-subgeneral position.
This finishes the induction. The claim thus is proved.

Now for every $v\in S$, let  $ F_{i,v}:= \mathrm{div}(s_{i, v}), i=0, \dots, l$, be the divisors on $X$ constructed above. 
Then, from the construction,   $F_{i,v} - N(\epsilon_{i,v}+\delta')D_{i,v}$ is effective for each $i=0, \dots, l$. Hence, by Proposition \ref{weil_eff} and (\ref{a1}), 
\begin{equation}\label{a2}
\lambda_{F_{i,v}, v}(x)\geq N(\epsilon_{i,v}+\delta')\lambda_{D_{i,v}, v}(x) + O_v(1)
\ge N\epsilon_{D_{i, v}}(A)\lambda_{D_{i,v}, v}(x) + O_v(1)
\end{equation}
for $x$ outside $\Supp F_{i,v}$ and $\Supp D_{i,v}$.

 Denote by  $\phi: X\rightarrow  \PP^{\tilde N}(k)$ the canonical embedding 
 associated to the very ample divisor $N(1+\delta)A$ 
 and let   $H_{1, v}, \dots, H_{l+1, v}$ be the hyperplanes in $\PP^{\tilde N}(k)$ with 
  $F_{j, v}=\phi^*H_{j+1, v}$ for $j=0,\dots,l$. Following notation of Lemma \ref{quang}, we denote $L_{1,v},\dots,L_{l+1,v}$ the linear forms defining $H_{1,v},\dots, H_{l+1,v}$ respectively.
By Lemma \ref{quang}, there exists hyperplanes $\hat H_{1,v},\dots, \hat H_{n+1,v}$ with defining linear forms $\hat L_{t, v}, t=1, \dots, n+1$, such that $\hat L_{1,v} = L_{1,v}$, $\hat L_{t, v} \in \text{span}_k(L_{2,v},\dots, L_{l-n+t,v})$ for $t=2,\dots,n+1$
and  $\phi^*\hat H_{1, v},\dots,\phi^*\hat H_{n+1, v}$ are located  in general position on $X$.  
Applying Theorem A to  $\phi^*\hat H_{1, v},\dots,\phi^*\hat H_{n+1, v}$, 
we conclude that there exists a Zariski-closed set $Z$ such that for all $x\in X(k)\setminus Z$,
\begin{equation}\label{B}
\sum_{v\in S}\sum_{i=1}^{n+1} \lambda_{\phi^*\hat H_{i,v}, v}(x) \leq [(n+1)+\delta]h_{N(1+\delta)A}(x).\end{equation}
On the other hand, fixing  $P=\phi(x)\in \PP^{\tilde N}(k)$, we reorder $H_{1,v},\dots, H_{l+1,v}$ such that $\|L_{1,v}(P)\|_v \leq  \|L_{2,v}(P)\|_v\leq \cdots 
\le  \|L_{l+1, v}(P)\|_v$.
For $t=2,\dots,n+1$, using the fact that $\hat L_{t, v} \in \text{span}_k(L_{2, v},\dots, L_{l-n+t, v})$, we have
$$\|\hat L_{t, v}(P)\|_v\leq C_v \max_{2\leq j\leq l-n+t}\|L_{j, v}(P)\|_v = \|L_{l-n-t, v}(P)\|_v$$
for some constant $C_v>0$. Thus, by the definition of  $\lambda_{\hat H_t,v}(P)$, 
$$
\lambda_{\hat H_{t,v}, v}(P) \geq \lambda_{H_{l-n-t,v}, v}(P)+O_v(1).$$
Therefore
\begin{equation}\label{a}
\begin{split}
\sum_{j=1}^{l+1}\lambda_{H_{j,v}, v}(P) 
&= \sum_{j=1}^{l-n+1}\lambda_{H_{j,v}, v}(P) + \sum_{j=l-n+2}^{l+1}\lambda_{H_{j,v}, v}(P) \\
&\leq \sum_{j=1}^{l-n+1}\lambda_{H_{j,v}, v}(P) + \sum_{t=2}^{n+1}\lambda_{\hat H_{t,v}, v}(P) +O_v(1) \\
&\leq (l-n+1)\lambda_{H_{1,v}, v}(P)+ \sum_{t=2}^{n+1}\lambda_{\hat H_{t,v}, v}(P) +O_v(1)\\
&\leq (l-n+1) \sum_{i=1}^{n+1}\lambda_{\hat H_{i,v}, v}(P) + O_v(1).
\end{split}
\end{equation}
Noticing that  $P=\phi(x)$ and using the fact that $\lambda_{H_{j,v}, v}(\phi(x))= \lambda_{\phi^*H_{j,v}, v}(x)=\lambda_{F_{j-1, v}, v}(x)$, it gives,
by combining (\ref{B}) and (\ref{a}), we obtain
\begin{equation}\label{c}
  \begin{split}
    \sum_{v\in S} \sum_{i=0}^{l} \lambda_{F_{i, v}, v}(x)&= \sum_{v\in S}
    \sum_{j=1}^{l+1}\lambda_{H_{j,v}, v}(P) \\
    &\leq  (l-n+1)  \sum_{v\in S} \sum_{i=1}^{n+1}\lambda_{\hat H_{i,v}, v}(P) + O(1)\\
    & \leq (l-n+1)[(n+1)+\delta]h_{N(1+\delta)A}(x).
\end{split}
 \end{equation}
This, together with (\ref{a2}), gives
\[
\begin{split}
N\sum_{v\in S}\sum_{j=0}^l \epsilon_{D_{j, v}}(A)\lambda_{D_{j,v}, v}(x)
  &\leq \sum_{v\in S} \sum_{i=0}^{l} \lambda_{F_{i, v}, v}(x)\\
  &\leq (l-n+1)[(n+1)+\delta]h_{N(1+\delta)A}(x) +O(1)
\end{split}
\]
for $x\in X(k)\setminus Z$.
Note that  $h_{N(1+\delta)A}(x) =N(1+\delta)h_A(x)$, hence we can rewrite the above inequality as
$$\sum_{v\in S}\sum_{j=0}^l \epsilon_{D_{j,v}}(A)\lambda_{D_{j,v},v}(x)\leq (l-n+1)(n+1+\delta)(1+\delta)h_A(x).$$
Recall our choice of $\delta$, we 
get for all $x\in X(k)\setminus Z$,
$$
\sum_{v\in S}\sum_{j=0}^l \epsilon_{D_{j, v}}(A) \lambda_{D_{j,v},v}(x)\leq [(l-n+1)(n+1)+\epsilon]h_{A}(x).$$
The theorem is thus proved.

Note that Lemma \ref{quang} also holds for the field ${\mathbb C}$ case.  So the above argument  together, with Theorem B to replace Theorem A (note that the statement of Theorem B is slightly different from Theorem A, but we can obtain a statement which exactly corresponds to Theorem B, see Theorem 1.8 in \cite{HL17}), which derives the following analytic result. 

\begin{theorem} [Analytic  Part] Let
 $X$ be a complex  projective variety of dimension $n$.
 Let $D_1, \dots, D_q$ be effective Cartier divisors located in $l$-subgeneral position on $X$ with $l\ge n$. Let $A$ be an ample Cartier divisor on $X$. 
 Let $f: \mathbb C \rightarrow X$ be a
 holomorphic map with Zariski dense image. Then, for $\epsilon>0$, 
$$\sum_{j=1}^q \epsilon_{D_j}(A) m_f(r, D_j) \leq_{exc} [(l-n+1)(n+1)+\epsilon]T_{f,A}(r).$$
\end{theorem}

\section{The  subscheme case}
We follow the notation in \cite{rw} or \cite{HL17}.
Let $Y$ be a closed subscheme on a projective variety $V$ defined over $k$.
Then one can associate to each place $v\in M_k$ a function
$$
\lambda_{Y,v}: V\setminus \Supp(Y)\to \mathbb R
$$
satisfying some functorial properties (up to a constant) described 
in \cite[Theorem 2.1]{Sil}.
Intuitively, for each $P\in V, v\in M_k$
$$
\lambda_{Y,v}( P)=-\log(v\text{-adic distance from $P$ to $Y$}).
$$
The starting point to study such arithmetic distance functions on subschemes is to understand the natural operations (i.e. addition, intersection, image and inverse image under morphism) on subschemes, which are parallel with the case of divisors. The key observation here is that we can identify a closed subscheme $Y$ of $V$ with its ideal sheaf $\SI_Y$, and such operations are naturally defined in terms of operations on ideal sheaves. We briefly summarize them as follows, details can be found in Section 2, \cite{Sil}. Let $Y,Z$ be closed subschemes of $V$.

(i) The sum of $Y$ and $Z$, denoted by $Y+Z$ is the subscheme of $V$ with ideal sheaf $\SI_Y \SI_Z$.

(ii) The intersection of $Y$ and $Z$, denoted by $Y\cap Z$, is the subscheme of $V$ with ideal sheaf $\SI_{Y}+\SI_{Z}$.

(iii) The union of $Y$ and $Z$, denoted by $Y\cup Z$, is the subscheme of $V$ with ideal sheaf $\SI_Y \cap \SI_Z$.

(iv) Let $\phi:W\to V$ be a morphism of varieties. The inverse image of $Y$ is the subscheme of $W$ with ideal sheaf $\phi^{-1}\SI_Y \cdot \SO_W$, denoted by $\phi^* Y$.

The following lemma indicates the existence of local Weil functions for closed subschemes: 
\begin{lemma} [Lemma 2.2 in \cite{Sil}]\label{representation}
Let $Y$ be a closed subscheme of $V$.  There exist effective Cartier divisors $D_1,\cdots,D_r$ such that 
$$
Y=\cap_{i=1}^r D_i.
$$
\end{lemma}
\begin{definition} \label{weil}Let $k$ be a number field, and $M_k$ be the set of places on $k$.  Let $V$ be a projective variety over $k$ and let $Y\subset V$ be a closed subscheme of $V$.
We define the (local) Weil function for $Y$ with respect to $v\in M_k$ as 
\begin{align}\label{WeilY}
\lambda_{Y, v}=\min_{1\leq i\leq r} \{\lambda_{D_i, v}\},
\end{align}
when $Y=\cap_{i=1}^r D_i$ (such $D_i$ exist according to the above lemma), where $\lambda_{D_i,v}$ is the Weil function for Cartier divisors appeared in previous section. Hence in particular if the closed subscheme is a Cartier divisor then the definition coincides with the one given in discussion of divisor case.
\end{definition}

 The behavior of Weil functions for closed subschemes have some similarity with the case of divisors in the sense as follows:
\begin{theorem}[Theorem 2.1 in \cite{Sil}] 
Let $Y,Z$ be closed subschemes of $V$. Then up to $O_v(1)$:

(i) $\lambda_{Y\cap Z, v} = \min \{\lambda_{Y,v},\lambda_{Z,v}\}.$

(ii) $\lambda_{Y+Z,v} = \lambda_{Y,v} +\lambda_{Z,v}.$

(iii) If $Y\subset Z$ then $\lambda_{Y,v}\leq \lambda_{Z,v}$.

Furthermore, let $\phi:W\to V$ be morphism of varieties, one has

(iv) $\lambda_{\phi^*Y,v}(P)=\lambda_{Y,v}(\phi(P))$ for $P\in W\setminus \supp\; \phi^*Y$. 

\begin{definition}
Let $X$ be a projective variety. $Y$ be a closed subscheme of $X$ corresponding to a coherent sheaf of ideals $\SI$. $\mathcal S$ be the sheaf of graded algebras $\bigoplus_{d\geq 0}\SI^d$ with convention $\SI^0 = \SO_X$ and $\SI^d$ be the $d$-th power of $\SI$. Then the scheme $\Proj \; \mathcal S$ is called the blowing-up of $X$ with respect to $\SI$ or blowing-up of $X$ along $Y$.
\end{definition}

\end{theorem}
 The functoriality of Weil function stated in (iv) of above theorem, is of particular importance when we deal with blowing-ups, hence we reformulate it as a lemma.
\begin{lemma}[Lemma 2.5.2 in \cite{Vojta_LNM}; Theorem 2.1(h) in \cite{Sil}] \label{new} Let $Y$ be a closed subscheme of $V$, and let ${\tilde V}$ be the blowing-up of $V$ along $Y$ with exceptional divisor $E$. Then 
 $\lambda_{Y,v}(\pi(P))=\lambda_{E, v}(P)+O_v(1)$ for $P\in {\tilde V}\setminus \Supp E$.
 \end{lemma}
 The strategy to deal with closed subschemes is to look at blowing-ups and relate to the case of divisors using functoriality of Weil functions and Height functions. However, to apply the method in divisor case we need to obtain linear equivalent ample Cartier divisors in $l$-subgeneral position, for which it's convenient to go back downstairs to avoid dealing with the influence of exceptional divisors. The following lemma allows us to find Cartier divisors downstairs with desired intersection property.

\begin{lemma}[See Lemma 5.4.24 in \cite{L}]\label{L}
  Let $X$ be projective variety,  $\SI$ be a coherent ideal sheaf.  Let $\pi:\tilde X\to X$ be the blowing-up of $\SI$ with exceptional divisor $E$.
  Then there exists an integer $p_0=p_0(\SI)$ with the property that if $p\ge p_0$,  then $\pi_*\SO_{\tilde X}(-pE) =\SI^p$, and moreover, for any divisor $D$ on $X$, 
  $$H^i(X, \SI^p(D))=H^i( \tilde X, \SO_{\tilde X}(\pi^*D-pE))$$ for all $i\ge 0$. 
 \end{lemma} 
  
 We now prove the theorem, the proof basically follows Heier-Levin's proof (Page 7, \cite{HL17}).

\noindent{\it Proof of The Main Theorem}.
Denote by  $\SI_{i,v}$ the ideal sheaf of $Y_{i,v}$,   $\pi_{i,v}:{\tilde X_{i,v}}\to X$ the blowing-up of $X$ along $Y_{i,v}$, and 
$ E_{i,v}$ the exceptional divisor  on $\tilde X_{i,v}$.
Fix real number $\epsilon > 0$.
 Choose rational number $\delta > 0$ such that 
 $$\delta(l-n+1)+\delta(l-n+1)(n+1+\delta)<\epsilon.$$ 

Then for small enough positive rational number $\delta'$ depending on $\delta$, 
$\delta\pi^* A - \delta' E_{i,v}$ is $\QQ$-ample on $\tilde{X}_{i,v}$ for all $i\in \{0,\dots,l\}$ and $v\in S$. 
 By the definition of Seshadri constant, there exists a rational number $\epsilon_{i, v}>0$ such that 
\begin{equation}\label{b1}
  \epsilon_{i, v}
  +\delta' \ge \epsilon_{Y_{i, v}}(A)
  \end{equation}
  and   $\pi_{i,v}^* A - \epsilon_{i,v}E_{i,v}$ is $\QQ$-nef on $\tilde X_{i,v}$ for all $0\leq i\leq l,v\in S$.   
     With such choices, we have   $(1+\delta)\pi^*_{i,v}A - (\epsilon_{i,v}+\delta')E_{i,v}$
is  ample $\QQ$-divisor on $\tilde X_{i,v}$ for all $i,v$. Let $N\gg 0$ be an integer large enough s.t. $N(1+\delta)\pi^*_{i,v}A$ and 
$N[(1+\delta)\pi^*_{i,v}A - (\epsilon_{i,v}+\delta')E_{i,v}]$ 
are very ample integral divisors on $\tilde X_{i,v}$  for all $0\leq i \leq l$ and $v\in S$. 

 Fix $v\in S$, like the special divisor case in the above section, we claim that we can 
     construct Cartier  divisors $F_{0,v},\dots, F_{l,v}$ on $X$ which are located in  $l$-subgeneral position,       such that 
$F_{i,v}\sim N(1+\delta)A$ and $\pi_{i,v}^* F_{i,v}\geq N(\epsilon_{i,v}+\delta')E_{i,v}$ on $\tilde{X}_{i,v}$ for all $0\leq i \leq l$.
This can be done inductively.  Assume $F_{0,v},\dots, F_{j-1,v}$ have been constructed so that 
$F_{i,v}\sim N(1+\delta)A$ and $\pi_{i,v}^* F_{i,v}\geq N(\epsilon_{i,v}+\delta')E_{i,v}$ on $\tilde{X}_{i,v}$ for all $0\leq i \leq j-1$,
  and that 
$F_{0,v}, \dots,  F_{j-1,v}, Y_{j, v}, \dots, Y_{l, v}$ are in $l$-subgeneral position on $X$ (for $j=0$, this reduces to the hypothesis that $Y_{0, v}, \dots, Y_{l, v}$ are 
in $l$-subgeneral position). 
 To find $F_{j,v}$, we let  $\tilde F_{i,v}^{(j)} =\pi_{j, v}^*F_{i, v}$, $i=0, \dots, j-1$, and $\tilde{Y}_{i, v}^{(j)}= \pi^*_{j, v} Y_{i, v}$ for $i=j+1, \dots, l$. 
  Since,  in particular, $F_{0,v},\dots, F_{j-1,v}, Y_{j+1,v},\dots, Y_{l,v}$ are in $l$-subgeneral position on $X$,  and
  by noticing that  $\pi_{j,v}^{-1}$ is isomorphism outside of $Y_{j,v}$,   we know that 
 $\tilde F_{0,v}^{(j)}, \dots, \tilde F_{j-1,v}^{(j)},  \tilde Y_{j+1,v}^{(j)}, \dots,  \tilde Y_{l,v}^{(j)}$
 are in $l$-subgeneral position on  $\tilde X_{j,v}$ outside of $E_{j,v}$.
 It is thus reduced to the construction in the divisors case, and by the argument in the divisors case, there are sections
 $$\tilde s_{j,v}\in H^0(\tilde{X}_{j,v},\SO_{\tilde X_{j,v}}(N((1+\delta)\pi^*_{j,v}A-(\epsilon_{j,v}+\delta')E_{j,v}))),$$
 such that $\tilde F_{0,v}^{(j)}, \dots, \tilde F_{j-1,v}^{(j)}, \mathrm{div}(\tilde s_{j,v}), \tilde  Y_{j+1,v}^{(j)}, \dots,  \tilde Y_{l,v}^{(j)}$
 are in $l$-subgeneral position on  $\tilde X_{j,v}$ outside of $E_{j,v}$, where we regard $H^0(\tilde{X}_{j,v},\SO_{\tilde X_{j,v}}(N((1+\delta)\pi^*_{j,v}A-(\epsilon_{j,v}+\delta')E_{j,v})))$
 as a subspace of $H^0(\tilde{X}_{j,v},\SO_{\tilde X_{j,v}}(N((1+\delta)\pi^*_{j,v}A))).$  
  On the other hand,  by Lemma \ref{L}, we have, for $N$ big enough, 
 $$
 H^0(X,\SO_X(N(1+\delta)A)\otimes \SI_{j,v}^{N(\epsilon_{j,v}+\delta')}) =
 H^0(\tilde{X}_{j,v},\SO_{\tilde X_{j,v}}(N((1+\delta)\pi^*_{j,v}A- (\epsilon_{j,v}+\delta')E_{j,v}))).$$
 Therefore there is an effective divisor $F_{j, v} \sim  N(1+\delta)A$ on $X$ such that  $ \mathrm{div}(\tilde s_{j,v})=\pi_{j, v}^*F_{j, v}$. 
 Since $\tilde s_{j,v}\in H^0(\tilde{X}_{j,v},\SO_{\tilde X_{j,v}}(N((1+\delta)\pi^*_{j,v}A-(\epsilon_{j,v}+\delta')E_{j,v}))),$
 we have $\pi_{j,v}^* F_{j,v}\geq N(\epsilon_{j,v}+\delta')E_{j,v}$ on $\tilde{X}_{j,v}$.
  To complete the induction, it remains to
show that $F_{0,v},\dots, F_{j-1,v}, F_{j, v}, Y_{j+1,v},\dots, Y_{l,v}$ are in $l$-subgeneral position on $X$. Since  
$\tilde F_{0,v}^{(j)}, \dots, \tilde F_{j-1,v}^{(j)},  \mathrm{div}(\tilde s_{j,v}), \tilde  Y_{j+1,v}^{(j)}, \dots,  \tilde Y_{l,v}^{(j)}$
 are in $l$-subgeneral position on  $\tilde X_{j,v}$ outside of $E_{j,v}$, 
and $\pi_{j, v}$  is an isomorphism above the complement of $Y_{j, v}$,  it is clear that
$F_{0,v},\dots,  F_{j-1, v}, F_{j, v}, Y_{j+1,v},\dots, Y_{l,v}$ are in $l$-subgeneral position on $X$ outside $Y_{j, v}$. 
The full statement now follows from combining this with the fact that $Y_{j, v}$  is in $l$-subgeneral position
with $F_{0,v},\dots, F_{j-1,v},  Y_{j+1,v},\dots, Y_{l,v}$. 
This meets the requirement. So the claim holds by induction.

Following the claim, for every $v\in S$, we get linearly equivalent Cartier divisors $F_{i,v}\sim N(1+\delta)A$, $ i=0,\dots, l$ in $l$-subgeneral position on $X$.
By the same way in deriving (\ref{c}), we get \begin{equation}\label{m}
\sum_{v\in S}\sum_{i=0}^l \lambda_{F_{i,v}, v}(x)\leq (l-n+1)[(n+1)+\delta]h_{N(1+\delta)A}(x)
\end{equation}
on $X(k)\setminus Z$ where $Z$ is a proper Zariski-closed subset of $X$.
To relate with $Y_{j,v}$, by the fact that $\pi^*_{j,v}F_{j,v}\geq N(\epsilon_{j,v}+\delta')E_{j,v}$ on $\tilde X_{i,v}$,  by  applying Proposition \ref{weil_eff},  Lemma \ref{new}, and (\ref{b1}), for all $P\in \tilde X_{j,v}\setminus \Supp E_{j,v}$,
\[
\begin{split}
    N\epsilon_{Y_{j, v}}(A) \lambda_{Y_{j,v}, v}(\pi_{j,v}(P))
        &\leq N(\epsilon_{j,v}+\delta')\lambda_{Y_{j,v}, v}(\pi_{j,v}(P))\\ 
        &=N(\epsilon_{j,v}+\delta')\lambda_{E_{j,v},v}(P) \\
        &\leq \lambda_{\pi_{j,v}^*F_{j,v},v}(P)=\lambda_{F_{j,v},v}(\pi_{j,v}(P)).
\end{split}
\]
This, together with (\ref{m}), gives
\begin{equation}
\begin{split}
N\sum_{v\in S} \sum_{j=0}^l \epsilon_{Y_{j, v}}(A) \lambda_{Y_{j,v}, v}(x)
 &\leq \sum_{v\in S}\sum_{j=0}^l \lambda_{F_{j,v}}(x) \\
                  &\leq(l-n+1)[(n+1)+\delta]h_{N(1+\delta)A}(x)\\                        
                  &=(l-n+1)[(n+1)+\delta]N(1+\delta)h_{A}(x)
\end{split}
\end{equation}
on $X(k)\setminus Z$.
By the choice of $\delta$, the result follows.

In the complex case, we use the notations in \cite{rw}. In particular, for a complex projective variety  $X$ and 
 a  holomorphic map $f: {\mathbb C}\rightarrow X$ with $f({\mathbb C})\not\subset Y$, we define
the proximity function 
$$m_f(r,Y) = \int_0^{2\pi} \lambda_Y (f(re^{i\theta})) \frac{d \theta}{2\pi}.$$
The above argument also works, by replacing Theorem A with Theorem B, we can prove the analytic result of the Main Theorem stated in the introduction section.


\end{document}